\documentclass[11pt,reqno]{amsart}
\usepackage{tikz}
\usepackage{subfig}
\usepackage{epstopdf}
\usepackage{epsfig}
\usepackage{math}
\graphicspath{{./figures/}}
\usepackage{tikz}
\usetikzlibrary{shapes,arrows}
\tikzstyle{decision} = [diamond, draw, fill=blue!20, 
    text width=4.5em, text badly centered, node distance=3cm, inner sep=0pt]
\tikzstyle{block} = [rectangle, draw, fill=blue!20, 
    text width=5em, text centered, rounded corners, minimum height=4em]
\tikzstyle{line} = [draw, -latex']
\tikzstyle{cloud} = [draw, ellipse,fill=red!20, node distance=3cm,
    minimum height=2em]

\usetikzlibrary{positioning}
\tikzset{main node/.style={circle,fill=blue!20,draw,minimum size=1cm,inner sep=0pt},  }

\begin{document}
\title{A Fast algorithm for Earth Mover's Distance based on optimal transport and $L_1$ type Regularization}

\email{wgangbo@math.ucla.edu}
\author[Li]{Wuchen Li}
\email{wcli@math.ucla.edu}
\author[Osher]{Stanley Osher}
\email{sjo@math.ucla.edu }
\author[Gangbo]{Wilfrid Gangbo}
\address{Department of Mathematics, University of California, Los Angeles.}

\thanks{This work is partially supported by ONR grants N000141410683, N000141210838 and DOE grant DE-SC00183838.}


\maketitle
\begin{abstract}
We propose a new algorithm to approximate the Earth Mover's distance (EMD). Our main idea is motivated by the theory of optimal transport, in which EMD can be reformulated as a familiar $L_1$ type minimization.
We use a regularization which gives us a unique solution for this $L_1$ type problem. The new regularized minimization is very similar to problems which have been solved in the fields of compressed sensing and image processing, where several fast methods are available. In this paper, we adopt a primal-dual algorithm designed there, which uses very simple updates at each iteration and is shown to converge very rapidly. 
Several numerical examples are provided.
\end{abstract}

\markboth{Li, Osher, Gangbo}{Fast algorithms for Earth Mover's distance  }
\section{Introduction}

In this paper we propose a new algorithm to approximate the Earth Mover's distance (EMD), which is motivated by the theory of optimal transport and methods related to those used in compressed sensing and image processing.

We begin by reviewing some well known facts. 
EMD, which is also named the Monge problem or the Wasserstein metric, plays a central role in many applications, including image processing, computer vision and statistics e.t.c \cite{ stats, M1, W1, Sol1}. 
The EMD is a particular metric defined on the probability space of a convex, compact set $\Omega\subset \mathbb{R}^d$. 
Given two probability densities $\rho^0$, $\rho^1$ in a probability set $\mathcal{P}(\Omega)$, where \begin{equation*}
\mathcal{P}(\Omega)=\{\rho(x)\in L^1(\Omega) ~:~\int_{\Omega}\rho(x)dx=1,\quad \rho(x)\geq 0\}\ .
\end{equation*} 
The EMD deals with following (linear) minimization problem
\begin{equation}\label{EMD}
\begin{split}
EMD(\rho^0, \rho^1):=\min_\pi\quad \int_{\Omega\times \Omega}d(x,y) \pi(x,y)dxdy %\quad:\quad \textrm{$\pi$ is a transportation plan between $\rho^0$, $\rho^1$ }\}\ .
\end{split}
\end{equation}
with the constraint that the joint measure (also called the transport function) $\pi(x,y)$ has $\rho^0(x)$ and $\rho^1(y)$ as marginals, i.e. 
\begin{equation*}
\int_{\Omega}\pi(x,y)dy=\rho^0(x)\ ,\quad \int_{\Omega}\pi(x,y)dx=\rho^1(y)\ , \quad \pi(x,y)\geq 0\ .
\end{equation*}
Here $d$ is a distance function on $\mathbb{R}^d$ which we call the ground metric. In this paper,  we consider the ground metric either be the Euclidean distance ($L_2$) \cite{BC1, BC2} or the Manhattan distance ($L_1$) \cite{Alg1}. I.e. $d(x,y):=\|x-y\|_2$ or $\|x-y\|_1$. 
We call \eqref{EMD} with the $L_1$, $L_2$ ground metric the EMD-$L_1$, EMD-$L_2$.

In recent years, \eqref{EMD} has been well studied by the theory of optimal transport \cite{am2006, Gangbo1, S2009, vil2003}. The theory (remarkably) points out that \eqref{EMD} is equivalent to a new minimization problem
\begin{equation}\label{W1}
EMD(\rho^0,\rho^1)=\inf_{m}\{ \int_{\Omega}\|m(x)\| dx~ :~\nabla \cdot m(x)+\rho^1(x)-\rho^0(x)=0 \}\ ,
\end{equation}
where $\|\cdot\|$ is $2$-norm ($L_2$ ground metric) or 1-norm ($L_1$ ground metric) in $\mathbb{R}^d$ and $m: ~\Omega\rightarrow \mathbb{R}^d$ is a flux vector satisfying a zero flux condition. I.e. $m(x)\cdot \nu(x)=0$, where $\nu(x)$ is the unit normal vector for the boundary of $\Omega$.

The minimization \eqref{W1} has an interesting fluid dynamics interpretation. 
It turns out that \eqref{W1} can be viewed as the following optimal control problem  
\begin{equation*}
\inf_{\rho, m}\{\int_0^1\int_{\Omega} \|m(t,x)\| dx dt~:~\frac{\partial \rho}{\partial t}+\nabla\cdot m=0\ ,\quad \rho(0)=\rho^0,\quad \rho(1)=\rho^1 \ \}\ ,%\hspace{1cm} (2b)
\end{equation*}
where the minimum is taken among all possible flux functions $m(t,x)$, such that the probability density function is moved continuously in time, from $\rho^0$ to $\rho^1$. 
The optimal control problem has many minimizers. One of them is such that $\rho(t,x)=t\rho^1+(1-t)\rho^0$. Here the flux function $m(t,x)$ does not depend on the time and $\nabla\cdot m=-\frac{\partial \rho}{\partial t}=\rho^0-\rho^1$, so that the control problem becomes \eqref{W1}, see \cite{bb, Mc, vil2003} for details related to EMD-$L_2$ and \cite{below1, below2} for similar optimal transport problems related to EMD-$L_1$.

The formulation \eqref{W1} has two benefits numerically. First, the dimension in \eqref{W1} is lower than the one in the original problem \eqref{EMD}. Suppose we discretize $\Omega$ by a grid with $N$ nodes. Since the unknown variable $\pi(x,y)$ in \eqref{EMD} is supported on $\Omega\times \Omega$ and $m(x)$ in \eqref{EMD} only depends on $\Omega$, the number of grid points used in \eqref{W1} is $N$ while the number needed in \eqref{EMD} is $N^2$. Second, \eqref{W1} is an $L_1$-type minimization problem, which shares its structure with many problems in compressed sensing and image processing. It is possible to borrow a very fast and simple algorithm used there to solve EMD, see e.g. \cite{Osher2,  Osher3, Osher1}.

In this paper, we propose a new algorithm for EMD, leveraging the structure of the formulation \eqref{W1}. The algorithm mainly uses a finite volume method to discretize the domain $\Omega$ and then applies the framework of primal-dual iterations designed in \cite{pock1, pock2}. We overcome the lack of strict convexity of \eqref{W1} (EMD-$L_1$) by a regularization.
Since the regularized minimization is a perturbation of a homogenous degree one minimization, our algorithm inherits all the benefits of the primal-dual algorithm:
First, we use a shrink operator at each step, which handles the sparsity easily, see e.g. \cite{Osher2};   
Second, the algorithm converges rapidly and each step involves very simple formulae. Thus the complexity of the algorithm is very low and the program is very simple.

EMD has been shown to be effective in applications \cite{com1} and many linear programming techniques have been proposed, see e.g. \cite{linear2, Alg1} and many references therein. 
Recently, the authors in \cite{BC1, BC2, Sol1} used the Alternating Direction Method of Multipliers (ADMM), which can also solve the EMD with a general Finsler ground metric. 
We use the primal-dual algorithm rather than ADMM. So we do not need to solve an elliptic problem (i.e. the inverting of a Laplacian) and every iteration is explicit. Our updates are quite simple while ADMM might need fewer iterations. In addition, it is clear that the ADMM is difficult to parallelize while ours is quite easy. This means, with a parallel computer, the algorithm we propose can be made much faster.
The primal-dual algorithm has been used before in optimal transport. The authors in \cite{MFG} use it to compute the stationary solution of mean field games. However, the problem we consider is totally different. 
The emphasis of this paper is that we design simple and fast algorithms for EMD-$L_1$, EMD-$L_2$.%Monge problems of particular but important shapes ($L_1$, $L_2$ ground metric).   

The outline of this paper is as follows. In section 2, we propose a primal-dual algorithm for \eqref{W1} on a uniform grid. We analyze the algorithm in section 3. Several numerical examples are presented in section 4.  
\section{Algorithm}
The EMD problem, as presented in \eqref{W1}, has similar structure to many homogenous degree one regularized  problems.
In this section we will use a finite volume discretization to approximate \eqref{W1}. The discretized problem becomes an $L_1$-type optimization with linear constraints, which allows us to apply the hybrid primal-dual method designed in \cite{pock1, pock2}. 

We begin with considering EMD-$L_2$. We shall consider a uniform lattice graph $G=(V, E)$ with spacing $\Delta x$ to discretize the spatial domain, where 
$V$ is the vertex set  
\begin{equation*}
V=\{1,2,\cdots, N\}\ ,
\end{equation*} 
and $E$ is the edge set. 
Here $i=(i_1, \cdots, i_d)\in V$ represents a point in $\mathbb{R}^d$. 

We consider a discrete probability set supported on all vertices:
\begin{equation*}
\mathcal{P}(G)=\{(p_i)_{i=1}^N\in \mathbb{R}^{N}\mid \sum_{i=1}^Np_i=1\ ,~p_i\geq 0\ ,~i\in V \}\ ,
\end{equation*}
where $p_i$ represents a probability at node $i$, i.e. $p_i=\int_{C_i} \rho(x)dx$, $C_i$ is a cube centered at $i$ with length $\Delta x$. So $\rho^0(x)$, $\rho^1(x)$ is approximated by $p^0=(p^0_i)_{i=1}^N$ and $p^1=(p^1_i)_{i=1}^N$.

We use two steps to consider the EMD on $\mathcal{P}(G)$. 
We first define a flux on a lattice. Denote a matrix $m=(m_{i+\frac{1}{2}})_{i=1}^N\in \mathbb{R}^{N\times d}$, where each component $m_{i+\frac{1}{2}}$ is a row vector in $\mathbb{R}^d$, i.e. 
$$m_{i+\frac{1}{2}}=(m_{i+\frac{1}{2}e_v})_{v=1}^d=(\int_{C_{i+\frac{1}{2}e_v}}m^v(x)dx)_{v=1}^d\ ,$$
where $e_v=(0,\cdots, \Delta x,\cdots, 0)^T$, $\Delta x$ is at the $v$-th column. In other words, if we denote $i=(i_1, \cdots, i_d)\in \mathbb{R}^d$ and $m(x)=(m^1(x), \cdots, m^d(x))$, then 
\begin{equation*}
m_{i+\frac{1}{2}e_v}\approx m^v(i_1,\cdots, i_{v-1}, i_v+\frac{1}{2}\Delta x, i_{v+1},\cdots, i_d)\Delta x^d\ .
\end{equation*}
We consider a zero flux condition. So if a point $i+\frac{1}{2}e_v$ is outside the domain $\Omega$, we let $m_{i+\frac{1}{2}e_v}=0$. Based on such a flux $m$, we define a discrete divergence operator 
$\textrm{div}_G(m):=(\textrm{div}_G(m_i))_{i=1}^N$, where
\begin{equation*}
\textrm{div}_G(m_i):=\frac{1}{\Delta x}\sum_{v=1}^d (m_{i+\frac{1}{2}e_v}-m_{i-\frac{1}{2}e_v})\ .
\end{equation*}
We next introduce the discrete cost functional %on a graph
\begin{equation*}
\|m\|:=\sum_{i=1}^N\|m_{i+\frac{1}{2}}\|_{2}=\sum_{i=1}^N \sqrt{\sum_{v=1}^d |m_{i+\frac{e_v}{2}}|^2}\ .
\end{equation*}

To summarize, \eqref{W1} forms an optimization problem 
\begin{equation*}
\begin{aligned}
& \underset{m}{\text{minimize}}
& &   \|m\|  \\
& \text{subject to}
& & \textrm{div}_G(m)+p^1-p^0=0\ ,%\quad i=1,\cdots, N\ .
\end{aligned}
\end{equation*}
  which can be written explicitly as 
\begin{equation}\label{W1new}
\begin{aligned}
& \underset{m}{\text{minimize}}
& &  \sum_{i=1}^N \sqrt{\sum_{v=1}^d |m_{i+\frac{e_v}{2}}|^2} \\
& \text{subject to}
& & \frac{1}{\Delta x}\sum_{v=1}^d (m_{i+\frac{1}{2}e_v}-m_{i-\frac{1}{2}e_v})+p_i^1-p_i^0=0\ , \quad i=1,\cdots, N,~ v=1,\cdots, d\ .
\end{aligned}
\end{equation}
%We can prove that the set of $m$ satisfying (3) is not empty. 
We observe that \eqref{W1new} is an optimization problem, which is very similar to some problems in compressed sensing and image processing e.g. \cite{Osher2}, whose cost functional is convex and whose constraints are linear. Thus we solve \eqref{W1new} by looking at its saddle point structure. Denote $\Phi=(\Phi_i)_{i=1}^N$ as \eqref{W1new}'s Lagrange multiplier, thus we have \begin{equation}\label{saddle}
\min_{m}\max_{\Phi} \quad L(m, \Phi):=\min_{m}\max_{\Phi} \quad \|m\|+\Phi^T(\textrm{div}_G(m)+p^1-p^0)\ .
\end{equation}

Saddle point problems, such as \eqref{saddle}, are well studied by the first order primal-dual algorithm \cite{pock1, pock2}. The iteration steps are as follows:
\begin{equation}\label{iteration}
\begin{cases}
m^{k+1}=&\arg\min_{m} \|m\|+(\Phi^k)^T \textrm{div}_G(m)+\frac{\|m-m^k\|^2_2}{2\mu} \ ;\\
\Phi^{k+1}=&\arg\max_{\Phi} \Phi^T \big(\textrm{div}_G(m^{k+1}+\theta (m^{k+1}-m^k))+p^1-p^0\big)-\frac{\|\Phi-\Phi^k\|^2_2}{2\tau} \ ,
\end{cases}
\end{equation}
where $\mu$, $\tau$ are two small step sizes, $\theta\in [0, 1]$ is a given parameter, $\|m-m^k\|^2_2=\sum_{i=1}^N\sum_{v=1}^d(m_{i+\frac{1}{2}e_v}-m_{i+\frac{1}{2}e_v}^k)^2$ and $\|\Phi-\Phi^k\|^2_2=\sum_{i=1}^N(\Phi_i-\Phi_i^k)^2$. 
These steps are alternating a gradient ascent in the dual variable $\Phi$ and a gradient descent in the primal variable $m$. 

It turns out that iteration \eqref{iteration} can be solved by simple explicit formulae. Since the unknown variable $m$, $\Phi$ is component-wise  separable in this problem, each of its components $m_{i+\frac{1}{2}}$, $\Phi_i$ can be independently obtained by solving \eqref{iteration}. 

First, notice 
\begin{equation*}
\begin{split}
& \min_{m} ~\|m\|+(\Phi^k)^T \textrm{div}_G(m)+\frac{\|m-m^k\|^2_2}{2\mu} \\
=&\min_{m} \sum_{i=1}^N \sum_{v=1}^d \sqrt{m_{i+\frac{e_v}{2}}^2}+\frac{1}{\Delta x}\sum_{i=1}^N \sum_{v=1}^d\Phi_i^k (m_{i+\frac{1}{2}e_v}-m_{i-\frac{1}{2}e_v}) +\frac{\|m-m^k\|^2_2}{2\mu}   \\
=&\sum_{i=1}^N\min_{m_{i+\frac{1}{2}}}\big(\|m_{i+\frac{1}{2}}\|_2-(\nabla_G \Phi_{i+\frac{1}{2}}^k)^Tm_{i+\frac{1}{2}}+\frac{1}{2\mu}\|m_{i+\frac{1}{2}}-m_{i+\frac{1}{2}}^k\|^2_2\big)\ ,
\end{split}
\end{equation*}
where $\nabla_G\Phi^k_{i+\frac{1}{2}}:=\frac{1}{\Delta x}(\Phi^k_{i+e_v}-\Phi_{i}^k)_{v=1}^d$. The first iteration in \eqref{iteration} has an explicit solution, which is: \begin{equation*}
m^{k+1}_{i+\frac{1}{2}}=\textrm{shrink}_2(m_{i+\frac{1}{2}}^k+\mu\nabla_G \Phi^k_{i+\frac{1}{2}}, \mu)\ ,
\end{equation*}
where we define the $\textrm{shrink}_2$ operation %on $\mathbb{R}^d$
\begin{equation*}
\textrm{shrink}_2(y, \alpha):=\frac{y}{\|y\|_2}\max\{\|y\|_2-\alpha, 0\}\ ,\quad \textrm{where $y\in \mathbb{R}^d$\ .}
%=\begin{cases}
% y-\alpha  \quad &\textrm{if $y>\alpha$\ ;}\\
%0  \quad &\textrm{if $-a\leq y\leq \alpha$\ ;}\\
%y+\alpha \quad & \textrm{if $y<-\alpha$\ .}
%\end{cases}
\end{equation*}

Second, consider
\begin{equation*}
\begin{split}
&\max_{\Phi} \Phi^T \textrm{div}_G(m^{k+1}+\theta (m^{k+1}-m^k))-\frac{\|\Phi-\Phi^k\|^2_2}{2\tau} \\
=& \sum_{i=1}^N \max_{\Phi_i}\big\{ \Phi_i [\textrm{div}_G(m_{i}^{k+1}+\theta(m^{k+1}_{i}-m^k_{i}))+p_i^1-p_i^0 ]-\frac{\|\Phi_i-\Phi^k_i\|^2_2}{2\tau}\big\}\ .\\
\end{split}
\end{equation*}
Thus the second iteration in \eqref{iteration} becomes 
\begin{equation*}
\Phi_i^{k+1}=\Phi_i^k+\tau \big\{ \textrm{div}_G(m_{i}^{k+1}+\theta(m^{k+1}_{i}-m^k_{i}))+p_i^1-p^0_i\}\ .
\end{equation*}
%\subsection{Algorithm}
We are now ready to state our algorithm.
\begin{tabbing}
aaaaa\= aaa \=aaa\=aaa\=aaa\=aaa=aaa\kill  
   \rule{\linewidth}{0.8pt}\\
   \noindent{\large\bf Primal-Dual for EMD-$L_2$}\\
  \1 \textbf{Input}: Discrete probabilities $p^0$, $p^1$; \\
    \3 Initial guess of $m^0$, step size $\mu$, $\tau$, $\theta\in[0,1]$.\\
  \1 \textbf{Output}: $m$ and EMD value $\|m\|$.\\
   \rule{\linewidth}{0.5pt}\\
1.  \1 for $k=1, 2, \cdots$ \qquad \textrm{Iterates until convergence}\\
2.  \2 $m^{k+1}_{i+\frac{1}{2}}=\textrm{shrink}_2(m_{i+\frac{1}{2}}^k+\mu\nabla_G \Phi^k_{i+\frac{1}{2}}, \mu)$ ;\\
3.  \2  $\Phi_i^{k+1}=\Phi_i^k+\tau \{\textrm{div}_G(m_{i}^{k+1}+\theta(m^{k+1}_{i}-m^k_{i}))+p_i^1-p^0_i\}$ ;\\ 
4.  \1 \End\\
   \rule{\linewidth}{0.8pt}
\end{tabbing}
%\begin{remark}
%%The $\textrm{shrink}_2$ operator we use is a multidimensional one, which is a projection to a ball, see e.g. \cite{Osher2}. In our next paper, we will use the conventional $\textrm{shrink}_1$, where the distance function is $\ell_1$.
%%\end{remark}
%In practice, we usually select $\theta=1$ in the algorithm. In next section, we shall prove under this selection, the algorithm converges to the minimizer of \eqref{W1new}.

We next consider EMD-$L_1$.
Similarly, \eqref{W1} forms the following optimization problem 
\begin{equation*}
\begin{aligned}
& \underset{m}{\text{minimize}}
& &   \|m\|_1 \\
& \text{subject to}
& & \textrm{div}_G(m)+p^1-p^0=0\ ,%,\quad i=1,\cdots, n\ .
\end{aligned}
\end{equation*}
  which can be written explicitly as 
\begin{equation}\label{W1new1}
\begin{aligned}
& \underset{m}{\text{minimize}}
& &  \sum_{i=1}^N\sum_{v=1}^d |m_{i+\frac{e_v}{2}}| \\
& \text{subject to}
& & \frac{1}{\Delta x}\sum_{v=1}^d (m_{i+\frac{e_v}{2}}-m_{i-\frac{e_v}{2}})+p_i^1-p_i^0=0\ , \quad i=1,\cdots, n\ ,~ v=1,\cdots, d\ .
\end{aligned}
\end{equation}
We observe that \eqref{W1new1} is an $L_1$ optimization problem, whose cost function is convex and whose constraints are linear.
However, the cost functional in \eqref{W1new1} is not strictly convex, which often implies the existence of multiple minimizers. To deal with this issue, we consider a small quadratic perturbation, through which we pick up a unique solution for a modified problem:
\begin{equation}\label{W1modify}
\begin{aligned}
& \underset{m}{\text{minimize}}
& &    \|m\|_1+\frac{\epsilon}{2}\|m\|_2^2 \\
& \text{subject to}
& & \textrm{div}_G(m)+p^1-p^0=0\ .
\end{aligned}
\end{equation}
Here $\|m\|_2^2=\sum_{i=1}^N\sum_{v=1}^d m_{i+\frac{e_v}{2}}^2$ and $\epsilon$ is a positive scalar. 

From now on, we solve \eqref{W1modify} by looking at its saddle point structure. Denote $\Phi=(\Phi_i)_{i=1}^N$ as \eqref{W1modify}'s Lagrange multiplier, we have \begin{equation}\label{saddle1}
\min_{m}\max_{\Phi} \quad L(m, \Phi):=\min_{m}\max_{\Phi}~\|m\|_1+\frac{\epsilon}{2}\|m\|_2^2+\Phi^T(\textrm{div}_G(m)+p^1-p^0)\ .
\end{equation}

Since $L(\cdot ,\Phi)$ is strictly convex which grows quadratically, and $L(m, \cdot)$ is linear, $L$ admits a saddle point solution.
Again, we solve \eqref{saddle1} by the first order primal-dual algorithm \cite{pock1, pock2}. The iteration steps are as follows:
\begin{equation}\label{iteration1}
\begin{cases}
m^{k+1}=&\arg\min_{m} \|m\|_1+\frac{\epsilon}{2}\|m\|_2^2+(\Phi^{k})^T \textrm{div}_G(m)+\frac{\|m-m^k\|^2_2}{2\mu} \ ;\\
\Phi^{k+1}=&\arg\max_{\Phi} \Phi^T \textrm{div}_G(m^{k+1}+\theta (m^{k+1}-m^k)+p^1_i-p^0_i)-\frac{\|\Phi-\Phi^k\|^2_2}{2\tau} \ .
\end{cases}
\end{equation}
%where $\|m-m^k\|^2_2=\sum_{i=1}^N\sum_{v=1}^d(m_{i+\frac{e_v}{2}}-m_{i+\frac{e_v}{2}}^k)^2$. 

As in the computation of EMD-$L_2$, we use simple exact formulae for \eqref{iteration1}. 
First, the update for $m^{k+1}$ has explicit solution, which acts separately on each component $m^{k+1}_{i+\frac{e_v}{2}}$:
\begin{equation*}
\begin{split}
& \min_{m} ~\|m\|_1+\frac{\epsilon}{2}\|m\|_2^2+(\Phi^k)^T \textrm{div}_G(m)+\frac{\|m-m^k\|^2_2}{2\mu} \\
=&\min_{m} \sum_{i=1}^N \sum_{v=1}^d \big\{|m_{i+\frac{e_v}{2}}|+\frac{\epsilon}{2}m_{i+\frac{e_v}{2}}^2+\frac{1}{\Delta x}\Phi_i^k (m_{i+\frac{e_v}{2}}-m_{i-\frac{e_v}{2}}) +\frac{(m_{i+\frac{e_v}{2}}-m^k_{i+\frac{e_v}{2}})^2}{2\mu}\big\}   \\
=&\sum_{i=1}^N\sum_{v=1}^d\min_{m_{i+\frac{e_v}{2}}}\big\{|m_{i+\frac{e_v}{2}}|+\frac{\epsilon}{2}m_{i+\frac{e_v}{2}}^2-\nabla_{G} \Phi^k_{i+\frac{e_v}{2}}m_{i+\frac{e_v}{2}}+\frac{1}{2\mu}(m_{i+\frac{e_v}{2}}-m_{i+\frac{e_v}{2}}^k)^2\big\}\ ,
\end{split}
\end{equation*}
where $\nabla_{G}\Phi^k_{i+\frac{e_v}{2}}:=\frac{1}{\Delta x}(\Phi^k_{i+e_v}-\Phi^k_{i})$. So the first iteration in \eqref{iteration1} has an explicit solution:
  \begin{equation*}
m^{k+1}_{i+\frac{e_v}{2}}=\frac{1}{1+\epsilon \mu}\textrm{shrink}(m_{i+\frac{e_v}{2}}^k+\mu\nabla_{G} \Phi^k_{i+\frac{e_v}{2}}, \mu)\ ,
\end{equation*}
where we define a shrink operation in $\mathbb{R}^1$ \begin{equation*}
\textrm{shrink}(y, \alpha):=\sign(y)\max\{|y|-\alpha, 0\}=\begin{cases}
 y-\alpha  \quad &\textrm{if $y>\alpha$\ ;}\\
0  \quad &\textrm{if $-a\leq y\leq \alpha$\ ;}\\
y+\alpha \quad & \textrm{if $y<-\alpha$\ .}
\end{cases}
\end{equation*}
Second, the update for $\Phi^{k+1}$ is same as the one in EMD-$L_2$. Since the second iteration in \eqref{iteration1} is identical to the one in \eqref{iteration}. 
\begin{tabbing}
aaaaa\= aaa \=aaa\=aaa\=aaa\=aaa=aaa\kill  
   \rule{\linewidth}{0.8pt}\\
   \noindent{\large\bf Primal-dual method for $\textrm{EMD}-L_1$}\\
  \1 \textbf{Input}: Discrete probabilities $p^0$, $p^1$; \\
    \3 Initial guess of $m^0$, parameter $\epsilon>0$, step size $\mu$, $\tau$, $\theta\in[0,1]$.\\
  \1 \textbf{Output}: $m$ and EMD value $\|m\|_1$.\\
   \rule{\linewidth}{0.5pt}\\
1.  \1 for $k=1, 2, \cdots$ \qquad \textrm{Iterates until convergence}\\
2.  \2 $m^{k+1}_{i+\frac{e_v}{2}}=\frac{1}{1+\epsilon \mu}\textrm{shrink}(m_{i+\frac{e_v}{2}}^k+\mu\nabla_{G} \Phi^k_{i+\frac{e_v}{2}}, \mu)$ ;\\
3.  \2 $\Phi_i^{k+1}=\Phi_i^k+\tau \{\textrm{div}_G(m_{i}^{k+1}+\theta(m^{k+1}_{i}-m^k_{i}))+p_i^1-p^0_i\}$ ;\\ 
4.  \1 \End\\
   \rule{\linewidth}{0.8pt}
\end{tabbing}
\begin{remark}
Here we use the conventional $\textrm{shrink}$ operator for EMD-$L_1$, while we apply what we call a $\textrm{shrink}_{2}$ operator for EMD-$L_2$.
\end{remark}

\section{Numerical analysis}
In this section, we prove that the primal-dual algorithm converges to the minimizer of our discretized minimization \eqref{W1new} or \eqref{W1new1}. 
For illustration, we prove the result for EMD-$L_2$. 
\begin{theorem}
Denote a linear operator $K~:~\mathbb{R}^{N\times d}\rightarrow \mathbb{R}^N$, such that 
\begin{equation*}
Km=(\textrm{div}_G(m_i))_{i=1}^N\ ,
\end{equation*}
and a saddle point of $L$ in \eqref{saddle} as $(m^*, \Phi^*)$. Choose $\theta=1$, $\tau\mu \|K\|_{\infty}^2<1$. Then $m^k$, $\Phi^k$ in iteration \eqref{iteration} converges to $m^*$, $\Phi^*$. 
Moreover, $\Phi^*$ satisfies 
\begin{equation}\label{eikonal}
\|\nabla_G \Phi_{i+\frac{1}{2}}^*\|_2=1\ ,\quad \textrm{if}\quad \|m^*_{i+\frac{1}{2}}\|_2>0\ .
\end{equation}
\end{theorem}
\begin{proof}
First, we only need to show that saddle point problem $L$ satisfies the condition of Theorem 1 in \cite{pock1, pock2}. We rewrite $L$ as 
\begin{equation*}
L(m, \Phi)=G(m)+\Phi^TKm-F(\Phi)\ ,
\end{equation*}
where $G(m)=\|m\|$, $Km=(\textrm{div}_G(m_i))_{i=1}^N$, and $F(\Phi)=\sum_{i=1}^N\Phi_i(p_i^0-p_i^1)$.
It is easy to observe that $G$, $F$ is a convex continuous function and $K$ is a linear operator. From Theorem 1 in \cite{pock1}, we prove the convergence result.

Second, since $\|m^*_{i+\frac{1}{2}}\|_2>0$, we have
\begin{equation*}
0=\frac{\partial L}{\partial m_{i+\frac{1}{2}}}|_{(m^*, \Phi^*)}=\frac{m^*_{i+\frac{1}{2}}}{\|m^*_{i+\frac{1}{2}}\|_2}-\nabla_G \Phi^*_{i+\frac{1}{2}}\ .
\end{equation*}
Thus
\begin{equation*}
1=\|\frac{m^*_{i+\frac{1}{2}}}{\|m^*_{i+\frac{1}{2}}\|_2}\|_2=\|\nabla_G \Phi^*_{i+\frac{1}{2}}\|_2\ .
\end{equation*}
We have proven $\Phi^*$ satisfies \eqref{eikonal}. %In all, we finish the proof.
\end{proof}

\begin{remark}
Theorem 1 holds similar for EMD-$L_1$ \eqref{W1new1}. Denote the saddle point of \eqref{W1new1} as ($m^*$, $\Phi^*$).
Then the scalar components of $\Phi^*$ satisfy  
\begin{equation*}\label{eikonal2}
|\nabla_{G}\Phi_{i+\frac{e_v}{2}}^*|=1\ ,\quad \textrm{if}\quad |m^*_{i+\frac{e_v}{2}}|>0\ .
\end{equation*}
Since if $|m_{i+\frac{e_v}{2}}|>0$, we have 
\begin{equation*}
0=\frac{\partial L}{\partial m_{i+\frac{e_v}{2}}}|_{(m^*,\Phi^*)}=\frac{m^*_{i+\frac{e_v}{2}}}{|m^*_{i+\frac{e_v}{2}}|}-\nabla_{G}\Phi_{i+\frac{e_v}{2}}^*\ .
\end{equation*}
Thus $$1=|\frac{m^*_{i+\frac{e_v}{2}}}{|m^*_{i+\frac{e_v}{2}}|}|=|\nabla_{G}\Phi^*_{i+\frac{e_v}{2}}|\ .$$
\end{remark}

From above convergence results, we are ready to show that the computational complexity for this primal-dual algorithm is $O(NM)$, where $M$ is the iteration number for a given error. 
It is known in \cite{pock1}, the algorithm converges with the rate of $O(\frac{1}{M})$. We also have the fact that each iteration has simple updates, which only need $O(N)$ operations. So our method requires overall $O(N)\times O(M)$ computations. 
In practice, we observe good performance, see the next section. This is because of the well known performance of shrink operations as observed in compressed sensing and image processing calculations.

It is also worth mentioning that if $\epsilon=0$, the cost functional $\|m\|_1$ in EMD-$L_1$ is not strictly convex, so 
there may exist multiple minimizers for problem \eqref{W1new1}. 
If $\epsilon>0$, the modified cost functional $\|m\|_1+\frac{\epsilon}{2}\|m\|_2^2$ is strongly convex, and we pick up a unique solution for the perturbed problem \eqref{W1new1}. In next section, we use numerical examples to demonstrate that such a unique solution approximates a particular minimizer of \eqref{W1new1} when $\epsilon$ is sufficient small. 

We do not claim that our discrete approximation converges as $\Delta x\rightarrow 0$ to the solution of (2). However, if as $\Delta x\rightarrow 0$ the family $m$ stays uniformly bounded, then it is easy to show that $m$ converges to a weak solution of (2).
\section{Examples}
In this section, we demonstrate several numerical results on a square $[-2, 2] \times [-2, 2]$. Our discretization used in the Figure 1-5 is a uniform $40\times 40$ lattice. The parameters are chosen as $\mu=\tau=0.025$, $\theta=1$. 
The initial flux $m$ and $\Phi$ are chosen as all zeros. We use the stopping criteria $$\frac{1}{N}\sum_{i=1}^N|\textrm{div}_G(m^k_{i})+p_i^1-p_i^0|\leq 10^{-5}\ .$$

\begin{figure}[H]
\centering
\subfloat[ $\rho^0$.]{\includegraphics[scale=0.25]{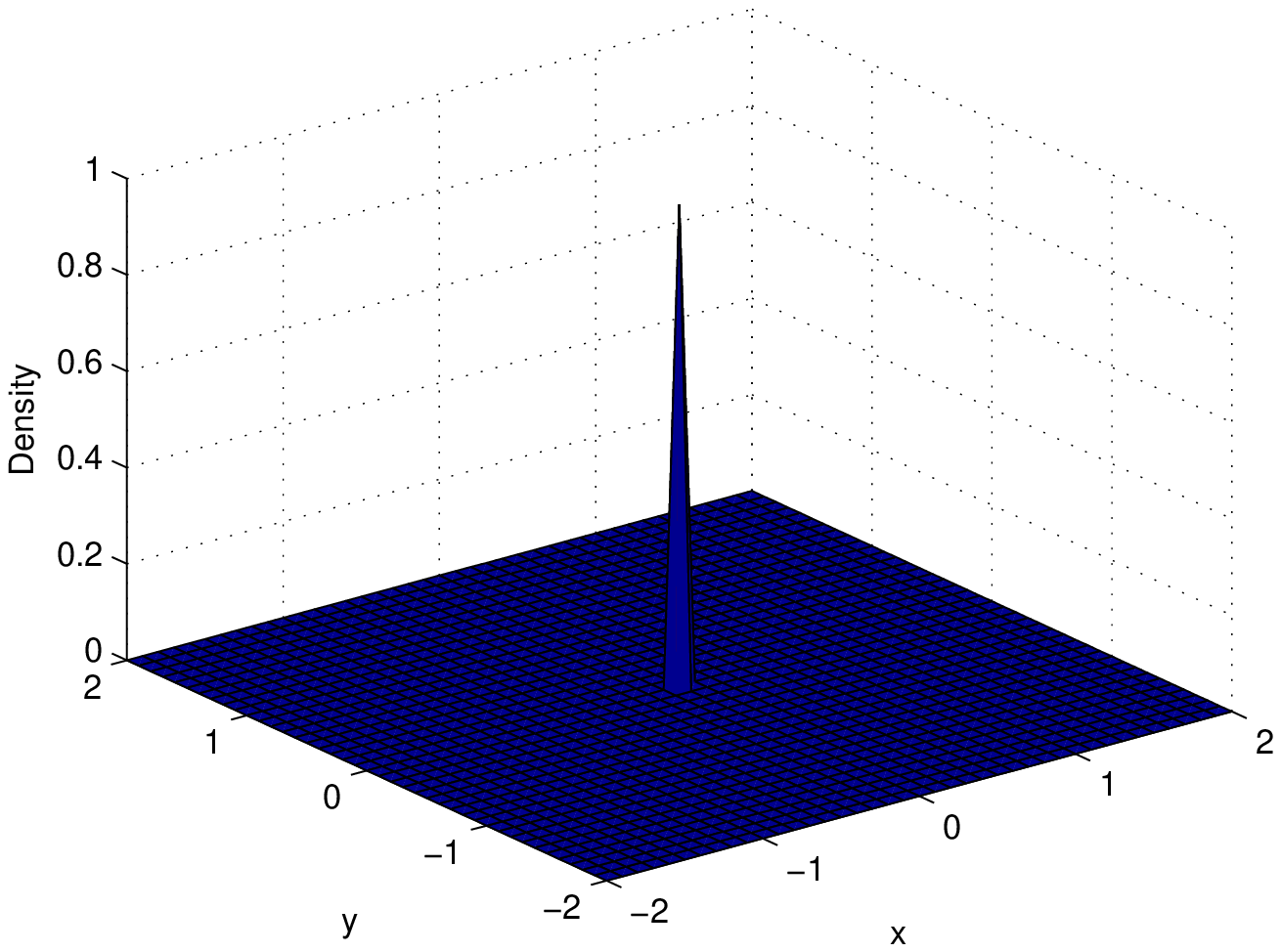}}\hspace{1cm}
\subfloat[ $\rho^1$.]{\includegraphics[scale=0.25]{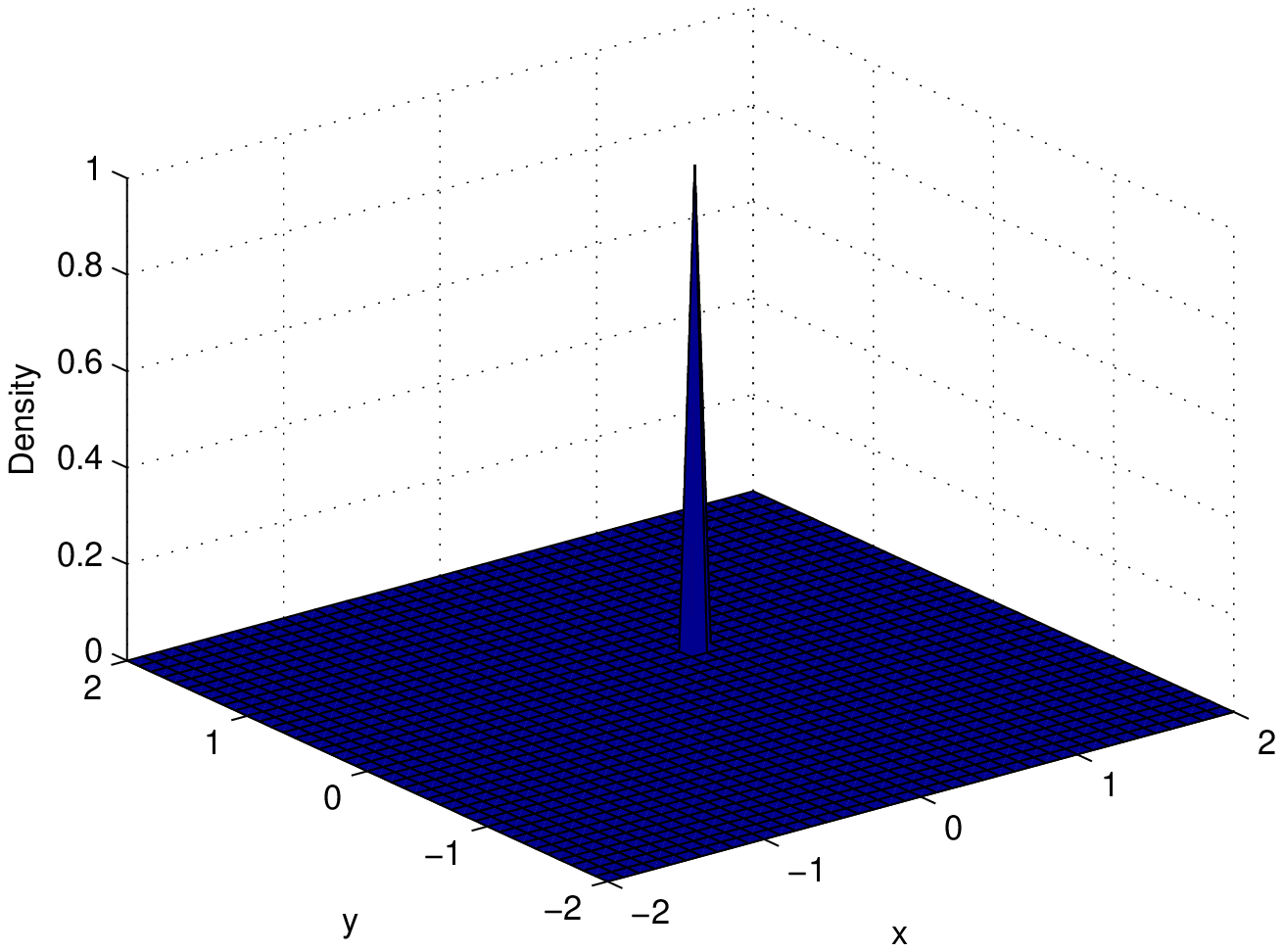}}
\caption{Here $\rho^0$ and $\rho^1$ are concentrated at $(0, 0)$, $(0.4, 0.4)$, i.e. $\rho^0=\delta_{(0, 0)}$, $\rho^1=\delta_{(0.4, 0.4)}$. 
The computed EMD-$L_1$, EMD-$L_2$ is $0.7981$, $0.6232$} \label{fig1}
\end{figure}

\begin{figure}[H]
\centering
\subfloat[ $\rho^0$.]{\includegraphics[scale=0.25]{1}}\hspace{1cm}
\subfloat[ $\rho^1$.]{\includegraphics[scale=0.25]{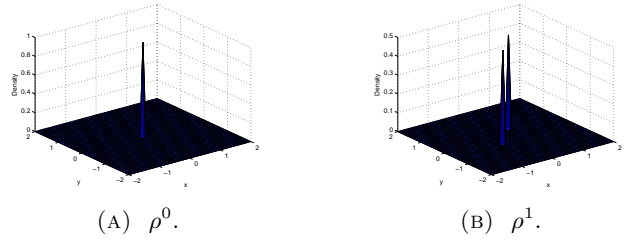}}
\caption{Here $\rho^0$ is concentrated at $(0,0)$, $\rho^1$ is concentrated at two positions, $(0.4,0.4)$ and $(-0.4, -0.4)$, 
i.e. $\rho^0=\delta_{(0,0)}$, $\rho^1=\frac{1}{2} (\delta_{(0.4, 0.4)}+\delta_{(-0.4, -0.4)})$. The computed EMD-$L_1$, EMD-$L_2$ is $0.8016$, $0.6232$.}
\end{figure}

\begin{figure}[H]
\centering
\subfloat[ $\rho^0$.]{\includegraphics[scale=0.25]{1}}\hspace{1cm}
\subfloat[ $\rho^1$.]{\includegraphics[scale=0.25]{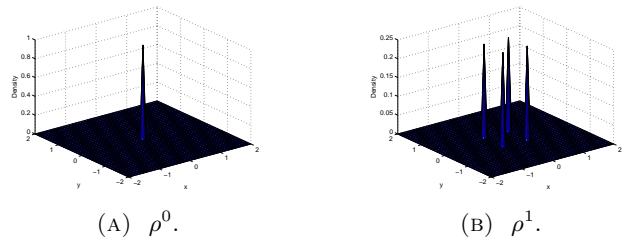}}
\caption{Here $\rho^0$ is concentrated at $(0,0)$, $\rho^1$ is concentrated at four positions, $(0.4,0.4)$, $(0.4, -0.4)$, $(-0.4, 0.4)$, $(-0.4, -0.4)$, 
i.e. $\rho^0=\delta_{(0,0)}$, $\rho^1=\frac{1}{4} (\delta_{(0.4, 0.4)}+\delta_{(0.4, -0.4)}+\delta_{(-0.4, 0.4)}+\delta_{(-0.4, -0.4)})$. The computed EMD-$L_1$, EMD-$L_2$ is $0.8002$, $0.5882$.}
\end{figure}

\begin{figure}[H]
\centering
\subfloat[ $\rho^0$.]{\includegraphics[scale=0.25]{1}}\hspace{1cm}
\subfloat[ $\rho^1$.]{\includegraphics[scale=0.25]{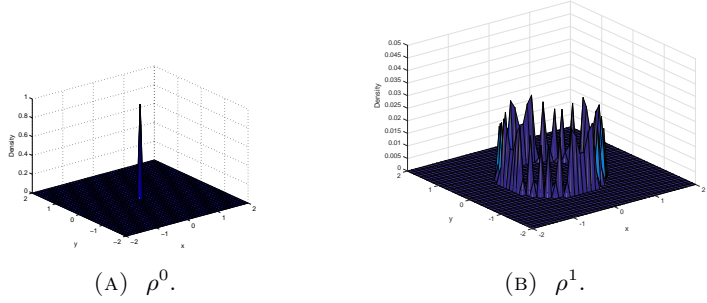}}
\caption{Here $\rho^0$ is concentrated at $(0, 0)$, $\rho^1$ is a measure supported on a circle, i.e.
 $\rho^0=\delta_{(0,0)}$, 
$\rho^1=\frac{1}{K} \big(e^{\frac{x^2+y^2}{\sigma}-\frac{(x^2+y^2)^2}{\sigma} }\big)$, where $K$ is a normalized constants and $\sigma=10^{-3}$.
The computed EMD-$L_1$, EMD-$L_2$ is $0.8794$, $0.6943$. }
\end{figure}

\begin{figure}[H]
\centering
\subfloat[ $\rho^0$.]{\includegraphics[scale=0.25]{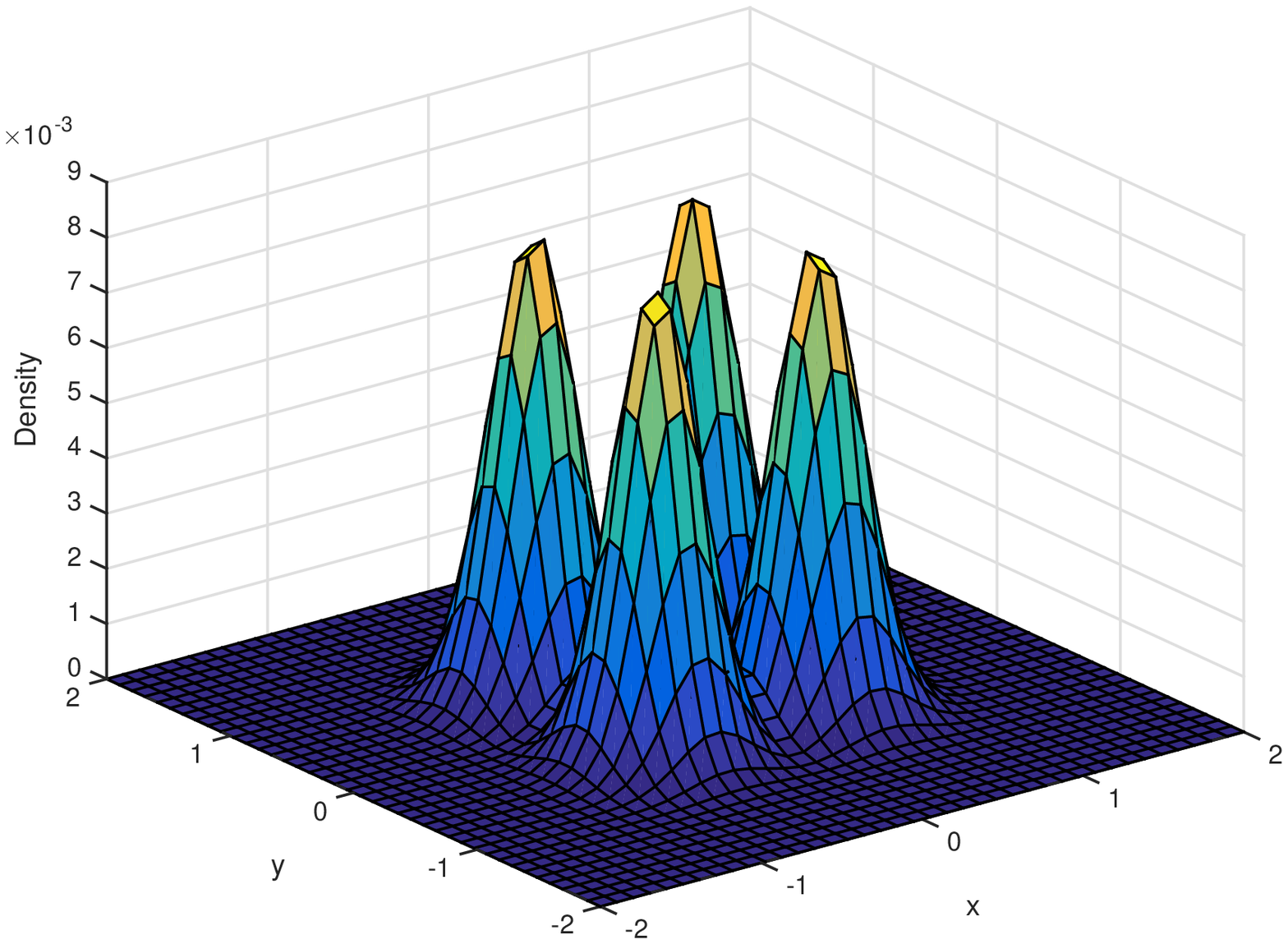}}\hspace{1cm}
\subfloat[ $\rho^1$.]{\includegraphics[scale=0.25]{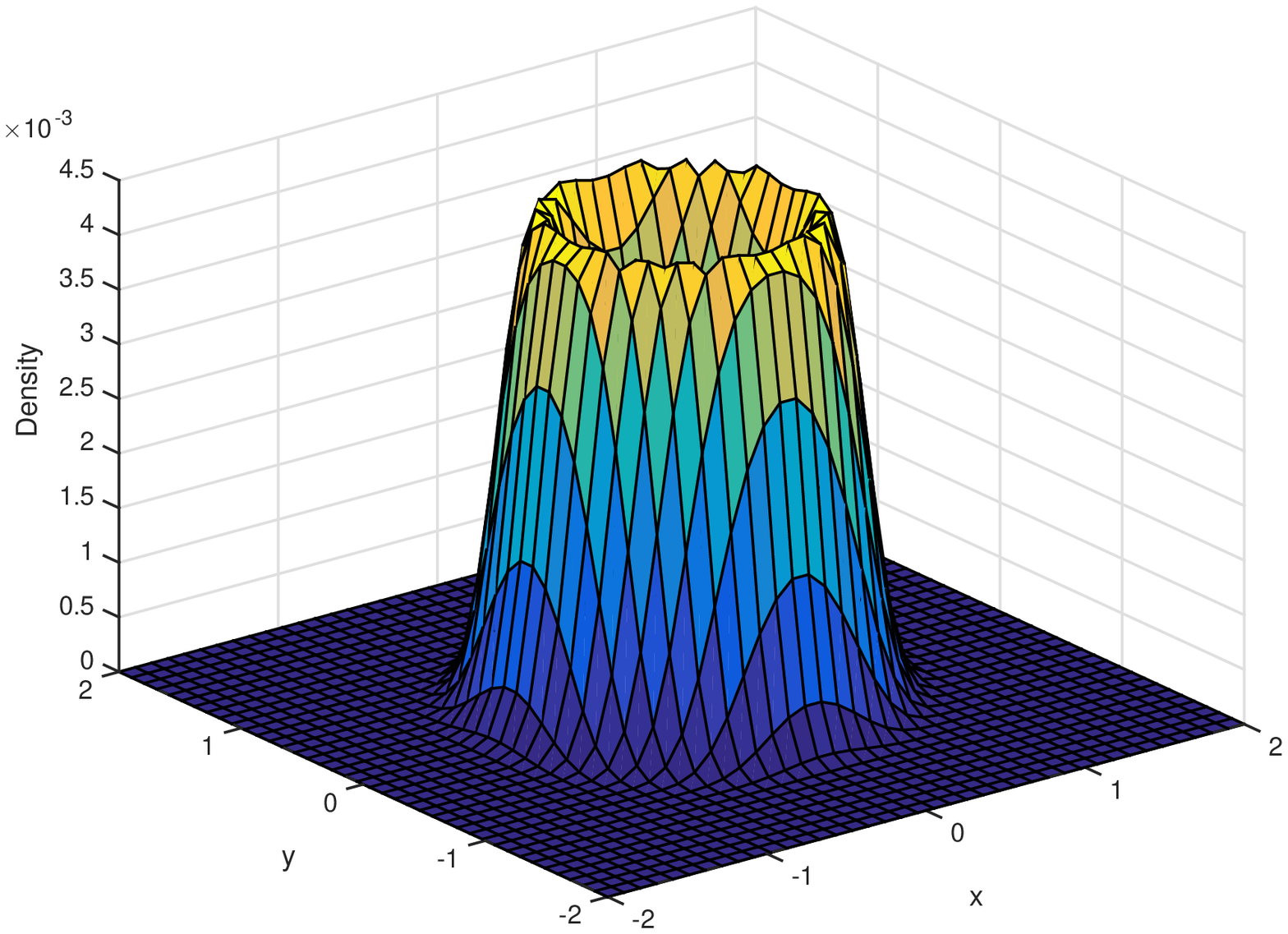}}
\caption{Here $\rho^0=\frac{1}{K_1} e^{-\frac{x^2+y^2-|x|-|y|}{\sigma}}$, 
$\rho^1=\frac{1}{K_2} \big(e^{\frac{x^2+y^2}{\sigma}-\frac{(x^2+y^2)^2}{\sigma} }\big)$, where $K_1$, $K_2$ are normalized constants and $\sigma=0.2$.
The computed EMD-$L_1$, EMD-$L_2$ is $0.1778$, $0.1259$. }
\end{figure}

Table 1, 2 reports the time under different grids for computing EMD-$L_1$, EMD-$L_2$ in Figure 1, 2, 3.
The implementation is done in MATLAB $2016$, on a 2.40 GHZ Intel Xeon processor with 4GB RAM.
\begin{table}
   \textbf{Example 1:}\\
    \begin{tabular}{  | l | l | l | p{4cm} |  }
    \hline
    Grids number (N) & Time (s)  &  Iteration  & Relative Error   \\ \hline
      $100$ &  0.0411 &  58 & 0.2071 \\ \hline
         $400$ & 0.4015 & 167 & 0.1495   \\ \hline
      $1600$ & 4.78 & 524&  0.1021 \\ \hline
   $6400$ &  64.81 & 1802&  0.0607\\  \hline
    \end{tabular}
 \\
 \textbf{Example 2:}\\
    \begin{tabular}{  | l | l | l | p{4cm} |  }
    \hline
    Grids number (N) & Time (s)  &  Iteration  & Relative Error   \\ \hline
      $100$ &  0.089 &  133& 0.2072 \\ \hline
         $400$ & 1.44 & 597 &   0.1497\\ \hline
      $1600$ & 8.41 & 901&    0.1014\\ \hline
   $6400$ &  118.8 & 3001 & 0.0596 \\  \hline
    \end{tabular}
    \\
  \textbf{Example 3:}\\
    \begin{tabular}{  | l | l | l | p{4cm} |  }
    \hline
    Grids number (N) & Time (s)  &  Iteration  & Relative Error   \\ \hline
      $100$ &  0.1334 & 210  & 0.0641 \\ \hline
         $400$ & 1.644 & 689 & 0.0536   \\ \hline
      $1600$ &12.27 & 1347&  0.0386  \\ \hline
   $6400$ &  130.37 & 3590 & 0.0199\\  \hline
    \end{tabular}

    \caption{We compute EMD-$L_2$ for Figure 1, 2, 3. Time is in seconds. The relative error is defined by 
$ \frac{|\|m\|-0.4\sqrt{2}|}{0.4\sqrt{2}}$, 
where $m$ is the computed minimizer of \eqref{W1new} and  $0.4\sqrt{2}$ is the analytical solution of EMD-$L_2$ (Euclidean distance between $(0.4, 0.4)$, $(0, 0)$).}
\end{table}

 \begin{table}
   \textbf{Example 1:}\\
    \begin{tabular}{  | l | l | l | p{4cm} |  }
    \hline
    Grids number (N) & Time (s)  &  Iteration  & Relative Error   \\ \hline
      $100$ &  $0.031$& $198$        &  $2.0\times 10^{-3}$ \\ \hline
         $400$ & $0.197$& $356$     &       $7.5\times 10^{-4}$   \\ \hline
      $1600$ &  $1.669$  & $786$  &    $2.7\times 10^{-4}$  \\ \hline
   $6400$ &   25.178 & $3057$ & $9.1\times 10^{-5}$ \\  \hline
    \end{tabular}
 \\ 
 \textbf{Example 2:}
 
    \begin{tabular}{  | l | l | l | p{4cm} |  }
    \hline
    Grids number (N) & Time (s)  &  Iteration  & Relative Error   \\ \hline
       $100$ &  0.047      &  156 &       $1.0\times 10^{-3}$ \\ \hline
         $400$ & 0.204 &  347  &     $3.8\times 10^{-4}$  \\ \hline
      $1600$ & 1.814   &850&    $1.3\times 10^{-4} $ \\ \hline
   $6400$ &   28.53  & 3483        &      $4.6\times 10^{-5}$ \\  \hline
    \end{tabular}
 \\ 
 \textbf{Example 3:}
 
    \begin{tabular}{  | l | l | l | p{4cm} |  }
    \hline
    Grids number (N) & Time (s)  &  Iteration  & Relative Error   \\ \hline
       $100$ &  0.039   &  142 &       $7.5\times 10^{-4}$ \\ \hline
         $400$ & 0.171  &  261  &     $2.6\times 10^{-4}$  \\ \hline
      $1600$ & 1.678   &803&    $8.9\times 10^{-5} $ \\ \hline
   $6400$ &    31.13   & 3792    &      $2.9\times 10^{-5}$ \\  \hline
    \end{tabular}

    \caption{ We compute EMD-$L_1$ with $\epsilon=0.01$ for Figure 1, 2, 3.
Here the stopping criteria is $\frac{1}{N}\sum_{i=1}^N|\textrm{div}_G(m^k_{i})+p_i^1-p_i^0|\leq 10^{-9}$. The relative error is defined by 
$ \frac{\|m^{\epsilon}\|_1+\epsilon \|m^{\epsilon}\|^2_2-0.8}{0.8}$, 
where $m^{\epsilon}$ is the computed minimizer of \eqref{W1modify} and  $0.8$ is the analytical solution of EMD-$L_1$ (Manhattan distance between $(0.4, 0.4)$, $(0, 0)$).}
\end{table}

We observe that the number of iterations is roughly $O(N)$ (sometimes less). So we claim that the complexity of our algorithm is at most $O(N^2)$, which roughly matches the result of the computation time in table 1 and 2.  
  \begin{table}[H]
    \begin{tabular}{  | l  | l | p{4cm} |  }
    \hline
    Grids number (N) & Time (s) EMD-$L_1$  &  Time (s) in EMD-$L_2$     \\ \hline
       $100$ &  0.0162& 0.1362   \\ \hline
         $400$ & 0.07529 &  1.645 \\ \hline
      $1600$ & 0.90 &  12.265 \\ \hline
   $6400$ & 22.38 & 130.37\\  \hline
    \end{tabular}
    \caption{For Figure 3, we compare the computation time for EMD-$L_1$ and EMD-$L_2$. Here we use the same stopping criteria: $\frac{1}{N}\sum_{i=1}^N|\textrm{div}_G(m^k_{i})+p_i^1-p_i^0|\leq 10^{-5}$.}
\end{table}

We also observe that we get approximately $10$ times faster speed for EMD-$L_1$ than EMD-$L_2$ in table 3. This is expected, since it is expensive to compute square roots for the Euclidean ground metric.
\begin{figure}[H]
\centering
\subfloat[  $\rho^0$.]{\includegraphics[scale=0.25]{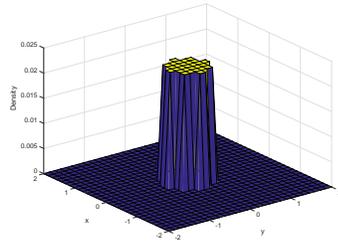}}\hspace{1cm}
\subfloat[  $\rho^1$.]{\includegraphics[scale=0.25]{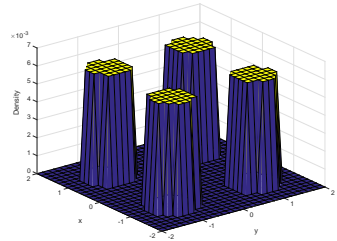}}\\
\subfloat[ Manhattan distance.]{\includegraphics[scale=0.3]{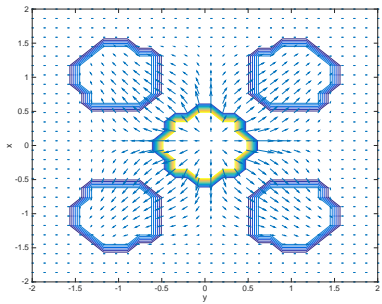}}\hspace{1cm}
\subfloat[ Euclidean distance.]{\includegraphics[scale=0.3]{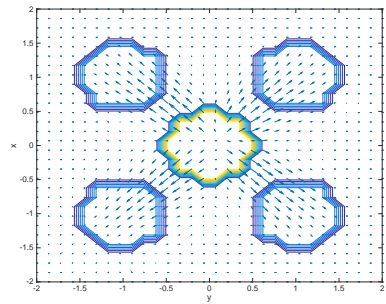}}
\caption{Comparison of minimizers $m(x)$ for EMD-$L_1$ and EMD-$L_2$. Here the initial measure is a uniform measure supported on a disk while the terminal measure is a uniform measure supported on four disjoint disks.}
\end{figure}

It is also worth mentioning that the quadratic perturbation in modified problem \eqref{W1modify} is necessary. We design the following two numerical results to demonstrate this. 
\begin{itemize}
\item[(i)] In table 4, we show that if we set $\epsilon=0$ in \eqref{W1modify}, the relative error can be enlarged if the total number of grids $N$ increases;
\item[(ii)] In table 5, we demonstrate that the minimizer of perturbed problem \eqref{W1modify} approximates one particular minimizer of \eqref{W1new} when $\epsilon$ approaches $0$. 
\end{itemize}

\begin{table}[H]
    \begin{tabular}{  | l  | p{3cm} |  }
    \hline  Grids number $N$   & Relative error \\ \hline
        400        &  $5.1\times 10^{-5}$    \\ \hline
         1600     &  $6.7\times 10^{-5}$   \\ \hline
         6400     &  $5.3\times 10^{-4}$   \\ \hline 
%       25600   &     \\ \hline 
    \end{tabular}
      \caption{We compute EMD-$L_1$ in Figure 3 with $\epsilon=0$ and different meshes. The terminal condition is $\frac{1}{N}\sum_{i=1}^N|\textrm{div}_G(m^k_{i})+p_i^1-p_i^0|\leq 10^{-6}$.
      The relative error is computed by $\frac{|\|m^0\|_1- 0.8|}{0.8} $, where $m^0$ is the computed minimizer of \eqref{W1modify}.}  
\end{table}

\begin{table}[H]
    \begin{tabular}{  | l  | p{3cm} |  }
    \hline
    $\epsilon$   & Relative error \\ \hline
              $0.1$& $9.1\times 10^{-4} $ \\ \hline 
         $0.01$         &  $9.4\times 10^{-5}$ \\ \hline
      $0.001$          & $1.6\times 10^{-5}$  \\ \hline
      $0.0001$        & $6.0\times 10^{-6}$  \\ \hline 
   %         $0.00001$& $1.7\times 10^{-7} $ \\ \hline 
    \end{tabular}
      \caption{We compute EMD-$L_1$ in Figure 3 with a fixed mesh and different values of $\epsilon$. The number of grid points is $N=1600$ and the terminal condition is $\frac{1}{N}\sum_{i=1}^N|\textrm{div}_G(m^k_{i})+p_i^1-p_i^0|\leq 10^{-6}$ .
      The relative error is computed by $ \frac{|\|m^\epsilon\|_1+\frac{\epsilon}{2}\|m^\epsilon\|_2^2- 0.8|}{0.8}$, where $m^\epsilon$ is the computed minimizer of \eqref{W1modify}.  }  
\end{table}

 \section{Conclusions}
To summarize, we applied a primal-dual algorithm to solve EMD with the $L_1$, $L_2$ ground metric. The algorithm inherits both key ideas in optimal transport theory and homogenous degree one regularized problems. Compared to current methods, our algorithm has following advantages:
\begin{itemize}
\item First, it leverages the structure of optimal transport, which transfers EMD into a $L_1$-type minimization. The new minimization contains only $N$ variables, which is much less than the original $N^2$ linear programming problem;
\item Second, it uses simple exact formulas at each iteration (including the shrink operator) and converges to a minimizer. 
\item Third, it will be very easy to parallelize and thus speed up the algorithm considerably.   
\end{itemize}

In addition, we consider a novel perturbed minimization
\begin{equation}\label{a}
\inf_{m}\{ \int_{\Omega}\|m(x)\|+\frac{\epsilon}{2}\|m\|_2^2 dx~ :~\nabla \cdot m(x)+\rho^1(x)-\rho^0(x)=0 \}\ ,
\end{equation}
to approximate EMD problem. Here $\|\cdot\|$ can be either the 1-norm or the 2-norm. 
In future work, we will study several theoretical properties of \eqref{a}, especially the relation between $m^\epsilon$ and $m$ when $\epsilon$ goes to $0$.

\end{document}